\documentclass[11pt,leqno]{article}

\usepackage{amssymb,amsmath,amsthm,mysects}

\newcommand{\n}{\noindent}

\newcommand{\bb}[1]{\mathbb{#1}}
\newcommand{\cl}[1]{\mathcal{#1}}
\newcommand{\sst}{\scriptstyle}

\theoremstyle{definition}
\newtheorem{defn}{Definition}

\theoremstyle{plain}

\newtheorem*{lm}{Lemma}

\newtheorem*{prop}{Proposition}
\newtheorem{thm}[defn]{Theorem}
\newtheorem{cor}[defn]{Corollary}

\theoremstyle{remark}
\newtheorem{rem}[defn]{Remark}

\begin{document}

\title{Completely co-bounded Schur multipliers}

\author{by\\
Gilles Pisier\footnote{Partially supported by NSF grant 0503688.}\\
Texas A\&M University\\
College Station, TX 77843, U. S. A.\\
and\\
Universit\'e Paris VI\\
Equipe d'Analyse, Case 186, 75252\\
Paris Cedex 05, France}

\date{}
\maketitle

\begin{abstract}  A linear map $u\colon \ E\to F$ between operator spaces
is called completely co-bounded if it is completely bounded as a map
from $E$ to the opposite of $F$. We give several simple results
about completely co-bounded Schur multipliers on $B(\ell_2)$ and
the Schatten class $S_p$. We also consider Herz-Schur
multipliers on groups.
\end{abstract}

\n In this short note, we wish to draw attention to
the notion of ``completely co-bounded' mapping between two operator spaces. Recall
that  an operator space can be defined as a Banach space  
$E$  given together with an isometric embedding $E\subset B(H)$
into the space $B(H)$   of all bounded operators on a Hilbert space $H$.
The theory of operator spaces started around 1987 with Ruan's thesis and has been considerably developed after that (notably by Effros-Ruan and Blecher-Paulsen, see \cite{ER,P4}), with applications mainly  to Operator Algebra Theory. In this theory, the morphisms between operator spaces
are the completely bounded maps (c.b. in short), defined as follows.
First note that if  $E\subset B(H)$ is any  subspace, then the space $M_n(E)$ of $n\times n$ matrices with entries in $E$ inherits the norm induced by $M_n(B(H))$.
The latter space is of course itself equipped with the norm of 
single operators acting naturally on $H\oplus\cdots\oplus H$ ($n$ times). Then,
a linear map $u\colon \ E\to F$ is called completely bounded (c.b.\ in short) if
\begin{equation}\label{eq11.0}
 \|u\|_{\text{cb}} \overset{\sst \text{def}}{=} \sup_{n\ge 1} \|u_n\colon \ M_n(E)\to M_n(F)\| < \infty
\end{equation}
where, for each $n\ge 1$, $u_n$ is defined by $u_n([a_{ij}]) = [u(a_{ij})]$. One denotes by $CB(E,F)$ the space of all such maps.

Given an operator space $E$,
the opposite $E^{op}$  is the same Banach space
as $E$, but equipped with the operator space structure (o.s.s. in short)
associated to any embedding $E\subset B(H)$
such that for any $a=[a_{ij}]\in M_n(E)$, we have
 $\|a\|_{M_n(B(H))}=\| [a_{ji}]  \|_{M_n(E)}$. 
 Thus $\|a\|_{M_n(E^{op})}=\| [a_{ji}]  \|_{M_n(E)}$.  It is easy to check that the (isometric linear) mapping
 $T\mapsto {}^{t}T \in B(H^*)$ realizes
 such an embedding (warning: here ${}^{t}T\colon\  H^*\to H^*$ designates the adjoint of $T$ in the Banach space sense).
 
 We call a map $u\colon \ E\to F$ between operator spaces   completely co-bounded if the same map
 is c.b. from $E$ to $F^{op}$.
 This definition is inspired by existing work on completely co-positive maps (cf. e.g. \cite{MM,M}
  and references there). I started to think about this notion after hearing   Marciniak's lecture on co-positive multipliers  at the 2004 Quantum probability conference in Bed\l ewo.

 While this definition seems at first glance a pointless variation,
 easy to reduce to the usual case,
 we hope in what follows to  convince the reader 
that it has a natural place in operator space theory and that it suggests   many interesting questions.
As a  first motivation for this notion, we should mention that the
non-commutative Grothendieck theorem, that came out of work  by the author and Haagerup, can be rephrased as saying that,
if $A,B$ are $C^*$-algebras,
any bounded linear mapping $u\colon \ A\to B^*$  is the sum
of a c.b. mapping and a co-c.b. one, see  \cite[p. 189]{PS} for details and more references.

\begin{defn}\label{compdefn1}
A linear map $u\colon \ E\to F$ between operator spaces will be called completely co-bounded if it is completely bounded as a mapping from $E$ into $F^{op}$ the opposite operator space. We then denote
\[
\|u\|_{cob} = \|u\colon \ E\to F^{op}\|_{cb}.
\]
\end{defn}

\begin{rem}\label{rem1} Obviously, $\|u\colon \ E\to F^{op}\|_{cb}=
\|u\colon \ E^{op}\to F\|_{cb}$
and  $(F^{op})^*={F^*}^{op}$ completely isometrically. Therefore, $u\colon \ E\to F$ is completely co-bounded iff
the same is true for $u^*$ and  $\|u\|_{cob}=\|u^*\|_{cob}$, since this is valid for
c.b. maps (cf. e.g. \cite{ER,P4}).
\end{rem}
\begin{rem}\label{rem2}Clearly if $B$ is a $C^*$-algebra with $F\subset B$ and if $\alpha\colon B\to B$ is an anti-automorphism
(for instance transposition on $B(\ell_2)$),
then $u\colon \ E\to F$ is completely co-bounded iff $\alpha u$ is c.b. and
$\|u\|_{cob} =\|\alpha u\|_{cb}$.
It is well known that the transposition on $M_n$ has c.b. norm equal to $n$ (cf. e.g. \cite[p. 418-419]{P4}).
Therefore, the identity map on $B(H)$ is \emph{not} completely co-bounded unless
$H$ is finite dimensional. More generally, the identity map on a von Neumann algebra $B$
is completely co-bounded iff  $B$ is of type $I_n$ for some finite $n$, i.e. iff $B$ is a  direct sum
of finitely many algebras of the form $M_n \otimes A_n$, with $A_n$ commutative.
\end{rem}

At first glance, the reader may have serious doubts for the need of the preceding notion ! but hopefully the next result will provide some justification.

\begin{thm}\label{compthm2}
A Schur multiplier $M_\varphi\colon \ [x_{ij}]\to [\varphi_{ij}x_{ij}]$ is completely co-bounded on $B(\ell_2)$ iff the matrix $[|\varphi_{ij}|]$ defines a bounded operator on $\ell_2$ and we have
\begin{equation}\label{compeq1}
\|M_\varphi\|_{cob} = \|TM_\varphi\|_{cb} = \|[|\varphi_{ij}|]\|_{B(\ell_2)}
\end{equation}
where $T\colon \ B(\ell_2)\to B(\ell_2)$ denotes the transposition. Moreover, if $\|M_\varphi\|_{cob}\le 1$, then $M_\varphi$ admits a factorization
\[
B(\ell_2) \overset{\sst J}{\longrightarrow} \ell_\infty({\bb N}\times {\bb N}) \overset{\sst M_\varphi}{\hbox to 25pt{\rightarrowfill}} B(\ell_2)
\]
where $J$ is the natural  inclusion map and where $\|M_\varphi\|_{cb}\le 1$.
\end{thm}

\begin{proof}
Assume that $[|\varphi_{ij}|]$ is in $B(\ell_2)$. Then the mapping
\begin{align*}
{\cl M}_\varphi\colon \ \ell_\infty({\bb N}\times {\bb N}) &\longrightarrow B(\ell_2)\\
[x_{ij}] &\longrightarrow [x_{ij}\varphi_{ij}]
\end{align*}
is obviously bounded with $\|{\cl M}_\varphi\| = \|[|\varphi_{ij}|]\|_{B(\ell_2)}$. Actually, more generally, if $x_{ij}\in B(H)$ with $\|x_{ij}\|\le 1$, then the matrix $[\varphi_{ij}x_{ij}]$ defines a bounded operator on $\ell_2(H)$ with norm easily seen to be majorized by $\|[|\varphi_{ij}|\|x_{ij}\|]\|_{B(\ell_2)}$ $\le \|[|\varphi_{ij}|]\|_{B(\ell_2)}$. This shows that $\|{\cl M}_\varphi\|_{cb} \le \|[|\varphi_{ij}|]\|_{B(\ell_2)}$. Let $J$ be as above. Clearly we have
\[
\|J\|_{cb} = 1
\]
and hence
\[
M_\varphi = {\cl M}_\varphi J
\]
with $\|{\cl M}_\varphi\|_{cb} = \|[|\varphi_{ij}|]\|_{B(\ell_2)}$. But since $\ell_\infty({\bb N}\times {\bb N})^{op}$ and $\ell_\infty({\bb N}\times {\bb N})$ are identical we can factorize $M_\varphi$ as follows
\[
M_\varphi\colon \ B(\ell_2) \overset{\sst J}{\longrightarrow} \ell_\infty({\bb N} \times {\bb N}) = \ell_\infty({\bb N}\times {\bb N})^{op} \overset{\sst {\cl M}_\varphi}{\hbox to 25pt{\rightarrowfill}} B(\ell_2)^{op}
\]
it follows that
\[
\|M_\varphi\colon \ B(\ell_2)\to B(\ell_2)^{op}\|_{cb}\le \|{\cl M}_\varphi\|_{cb} \le \|[|\varphi_{ij}|]\|_{B(\ell_2)}.
\]
This proves the ``if'' part.

Conversely, assume that $M_\varphi$ is completely co-bounded with $\|M_\varphi\|_{cob} \le 1$.  Let $x= [x_{ij}]$ be an $n\times n$ matrix viewed as sitting in $B(\ell_2)$. Let $B = B(\ell_2)$. Note that
\[
\left\|\sum^n_{ij=1} e_{ij}\otimes e_{ij}x_{ij}\right\|_{M_n(B)} = \|x\|_B
\]
while
\[
\left\|\sum^n_{ij=1} e_{ij}\otimes e_{ij}x_{ij}\right\|_{M_n(B^{op})} = \left\|\sum^n_{ij=1} e_{ji}\otimes e_{ij}x_{ij}\right\|_{M_n(B)} = \sup|x_{ij}|.
\]
By definition of $\|M_\varphi\|_{cob}\le 1$, we have
\[\left\|\sum e_{ij}\otimes e_{ij}\varphi_{ij}x_{ij}\right\|_{M_n(B)}=
\left\|\sum e_{ji}\otimes e_{ij}\varphi_{ij}x_{ij}\right\|_{M_n(B^{op})} \le \left\|\sum e_{ji}\otimes e_{ij}x_{ij}\right\|_{M_n(B)}
\]
which yields
\[
\|[\varphi_{ij}x_{ij}1_{\{i,j\le n\}}]\|_B \le \sup_{ij}|x_{ij}|\le 1.
\]
This implies
\[
\|[|\varphi_{ij}|]_{ij\le n}\|\le 1
\]
and since $n$ is arbitrary we obtain
\[
\|[|\varphi_{ij}|]\|_B \le 1.
\]
This proves the ``only if'' part.
The proof also yields \eqref{compeq1}.
\end{proof}

\begin{cor}\label{compcor3}
A Schur multiplier $M_\varphi$ is completely co-bounded
on $B(\ell_2)$ iff it factors through a commutative $C^*$-algebra or iff it factors through a minimal operator space and the corresponding factorization norm coincides with $\|M_\varphi\|_{cob}$.
\end{cor}

\begin{proof}
For any commutative $C^*$-algebra $C$ or for any $E\subset C$, we have clearly $E=E^{op}$, so a $cb$-factorization $M_\varphi\colon \ B \overset{\sst u_1}{\longrightarrow} E \overset{\sst u_2}{\longrightarrow} B$ yields
\[
\|M_\varphi\colon \ B\to B^{op}\|_{cb}\le \inf\{\|u_1\|_{cb} \|u_2\|_{cb}\}
\]
where the infimum runs over all possible factorizations. Conversely, the preceding shows the converse with a factorization through $\ell_\infty({\bb N}\times {\bb N})$.
\end{proof}
 \begin{rem}  
Let $G$ be an infinite discrete group.  Consider a 
function $f\colon\ G\to \bb C$, and
 the function $\hat f$
 defined on $G\times G$ by $\hat f(s,t)=f(st^{-1})$. 
By the well known Kesten-Hulanicki criterion (cf. e.g. \cite[Th. 2.4]{P3}) 
 $G$ is amenable iff there is a constant
 $C$ such that for any finitely supported $f$ we have
 $\sum\nolimits_{t\in G} |f{(t)}|\le C \|[|\hat f(s,t)|]\|_{B(\ell_2(G))}$
 and when $G$ is amenable this holds with $C=1$. Thus,
 by Theorem \ref{compthm2},  the inequality
 $$  \sum\nolimits_{t\in G} |f{(t)}|\le C \|M_{\hat f}\|_{cob} $$
 characterizes amenable groups. This should be compared with Bo\.zejko's  and Wysoczanski's criteria described in \cite[p. 54]{P3}
 and  \cite[p. 38]{P3}.
   \end{rem}

We now generalize Theorem \ref{compthm2} to the Schur multipliers that are bounded on the Schatten $p$-class $S_p$. We assume $S_p$ equipped with  the ``natural" operator space structure introduced in \cite{P}
using the complex interpolation method. We will use freely the notation and results from  \cite{P}.

\begin{thm}\label{compthm4}
Let $2\le p\le \infty$. Let $T\colon \ S_p\to S_p$ denote again the transposition mapping $x\to {}^tx$. Then a bounded Schur multiplier $M_\varphi\colon \ S_p\to S_p$ is completely co-bounded iff it admits a factorization as follows:
\[
S_p \overset{\sst J_p}{\hbox to 25pt{\rightarrowfill}} \ell_p({\bb N}\times {\bb N}) \overset{\sst {\cl M}_\varphi}{\hbox to 25pt{\rightarrowfill}} S_p
\]
where $J_p$ is the natural (completely contractive) inclusion and where $\|M_\varphi\|_{cob} = \|{\cl M}_\varphi\|_{cb}$.
\end{thm}
\begin{proof}
Note that the fact that $J_p\colon \ S_p\to\ell_p({\bb N}\times {\bb N})$ is completely contractive is immediate by interpolation between the cases $p=2$ and $p=\infty$.\\
The proof can then be completed following the same idea as for Theorem \ref{compthm2}. We have for any $[x_{ij}]$ in $M_n(S_p)$
\[
\left\|\sum e_{ij} \otimes e_{ji}\otimes x_{ij}\right\|_{S^n_p[S^n_p[S_p]]} = \left(\sum_{ij} \|x_{ij}\|^p_{S_p}\right)^{1/p}
\]
while
\[
\left\|\sum e_{ij}\otimes e_{ij}\otimes x_{ij}\right\|_{S^n_p[S^n_p[S_p]]} = \|[x_{ij}]\|_{S^n_p[S_p]}.
\]
Both of these identities can be proved by routine interpolation arguments starting from $p=\infty$ and $p=2$. This gives us
\[
\|[\varphi_{ij}x_{ij}]\|_{S^n_p[S_p]} \le \|M_\varphi\|_{cob} \left(\sum\|x_{ij}\|^p_{S_p}\right)^{1/p}
\]
which means (cf.\ \cite{P}) that
\[
\|{\cl M}_\varphi\|_{cb}\le \|M_\varphi\|_{cob}.
\]
To prove the converse, it suffices to notice again that
\[
\ell_p({\bb N}\times {\bb N})^{op} = \ell_p({\bb N}\times {\bb N}).\qquad \qed
\]
\renewcommand{\qed}{}\end{proof}
\begin{rem} Consider Schur multipliers from
$B(\ell_2)$ (or the subalgebra
of compact operators $K$)  into the trace class $S_1$. We refer to \cite{PS} for a detailed discusion of when such a multiplier
is bounded and when it is c.b. From that discussion
follows easily that such a multiplier is completely co-bounded
iff it is bounded. Indeed, more generally (see \cite{PS,HM2}), if $A,B$ are $C^*$-algebras 
a linear map $u\colon\ A\to B^*$ is completely co-bounded iff
there are a constant $c$ and states $f_1,f_2$, $g_1,g_2$ on $A,B$ respectively, such that for any $(a,b)\in A\times B$
\begin{equation}\label{eq10}|\langle u(a),b\rangle |\le c\left(( f_1(a^*a)g_1(b^*b))^{1/2}+(f_2(aa^*)g_2(bb^*))^{1/2}\right).\end{equation}
Consider a bounded Schur multiplier $\varphi=[\varphi_{ij}]$
from $B(\ell_2)$ (or  $K$)   to   $S_1$,
where $S_1$ is equipped with its natural
o.s.s. as the dual of $K$. By this we mean
that $\langle M_\varphi(a),b\rangle=\sum \varphi_{ij} a_{ij}b_{ij}$. Then (see \cite{PS})
 there are two nonnegative summable sequences $(\lambda_i)$
and $(\mu_i)$   such that for any $i,j$
$$|\varphi_{ij}|\le \lambda_i+\mu_j.$$
Then, using the Cauchy-Schwarz inequality,  we obtain \eqref{eq10}
with the states $f_1=g_1=(\sum \lambda_j )^{-1}\sum \lambda_j e_{jj}$
and $f_2=g_2=(\sum \mu_j )^{-1}\sum \mu_j e_{jj}$. This shows that $M_\varphi$ is automatically completely co-bounded. See \cite{X2}
for an extension to Schur multipliers
from $S_p$ to $S_{p'}$ with $2< p< \infty$.

 \end{rem} 
\begin{rem} The preceding two theorems illustrate the following simple observation: Assume that a linear map $u\colon \ E  \to F$ is both c.b. and  
completely co-bounded, then  $u$ can be completely boundedly factorized through an operator space $G$ for which the identity map
$I_G$
is completely co-bounded with $\|I_G\|_{cob}=1$. Indeed, 
we just consider  $G=F \cap F^{op}$ in the sense of \cite{P}
(this means $G=F$ equipped with the o.s.s. induced by the diagonal  embedding $F\subset F\oplus F^{op}$), then $u$ can be viewed
as $u\colon \ E\to G \to F$, with $\|u\colon \ E\to G\|_{cb}=\max\{ \|u\|_{cb}, \|u\|_{cob}\}$ and $\|G \to F\|_{cb}=1$.\\
Conversely, any mapping of the form
 $u\colon\ E\stackrel{v}\to G\stackrel{w  }\to F$,
with c.b. maps $v,w$   and $G$ such that  the identity $I=I_G$ on $G$
is completely co-bounded, must be both c.b. and
completely co-bounded (since
$u\colon\ E\stackrel{v}\to G\stackrel{I}\to G^{op}\stackrel{w  }\to F^{op}$ is c.b.).\\
Let us say that an operator space $G$ is self-transposed
if $I_G$ is completely co-bounded. This property passes obviously
to subspaces, quotients (and hence subquotients) and dual spaces.
It is also stable under ultraproducts.
Examples include any commutative $C^*$-algebra (or 
 any minimal operator space), 
by duality   any $L_1$-space (or any maximal 
operator space) and by interpolation    any $L_p$-space
($1\le p\le \infty$). Perhaps there is a nice characterization of
self-transposed operator spaces?
\end{rem}

Let $G$ be a finite group. Let $f,g\colon \ G\to {\bb C}$ be functions on $G$ and let $\lambda(f)\colon\ x\to f*x$ and $\rho(g)\colon \ x\to x*g$ be the associated convolutors on $\ell_2(G)$.

The Fourier transform of $f$ is defined as follows: \ for any irreducible representation $\pi$ on $G$ (i.e.\ $\pi\in\widehat G$) 
\[
\hat f(\pi) = \int f(t) \pi(t)^* dm(t)
\]
where $m$ is the normalized Haar measure on $G$. We have then

\begin{prop}
With the above notation, we have 
\[
\|\lambda(f)\rho(g)\|_{cob} = \sup\nolimits_{\pi\in\widehat G} \|\hat f(\pi)\|_2 \|\hat g(\pi)\|_2
\]
where $\|~~\|_2$ denotes the Hilbert--Schmidt norm on $H_\pi$. In particular,
\[
\|Id\colon \ C^*(G)\to C^*(G)\|_{cob} = \sup\nolimits_{\pi\in\widehat G} \dim(\pi).
\]
\end{prop}

\begin{proof}
Passing to Fourier transforms, we see that $\lambda(f) \rho(g)$ coincides with $\bigoplus\limits_{\pi\in\widehat G} L(\hat f(\pi)) R(\hat g(\pi))$ where $L(a)$ (resp.\ $R(a)$) denotes left (resp.\ right) multiplication by $a$ on $H_\pi$. Thus the result follows from the next lemma. 
\end{proof}

\begin{lm}
Let $H$ be a Hilbert space and let $u\colon \ B(H)\to B(H)$ be defined by $u(x) = a x b$. Then $u$ is completely co-bounded iff $a,b$ are both Hilbert--Schmidt operators and $\|u\|_{cob} = \|a\|_2 \|b\|_2$.
\end{lm}

\begin{proof}
We may easily reduce this to the finite dimensional case. So we assume $B(H) = M_n$. Then again we can write
\[
\left\|\sum e_{ij}\otimes ae_{ij}b\right\| \le \|u\|_{cob} \left\|\sum e_{ji}\otimes e_{ij}\right\| = \|u\|_{cob}.
\]\
But now, $T = \sum e_{ij}\otimes ae_{ij}b$ is a rank one operator on $\ell^n_2 \otimes \ell^n_2 = \ell^n_2(H)$ with $H=\ell^n_2$. Indeed, for any $h= (h_i)$, $k=(k_j)\in \ell^n_2(H)$ we have:
\begin{align*}
\langle Th,k\rangle &= \sum_{ij} \langle ae_{ij}bh_j,k_i\rangle\\
&= \sum_i \langle ae_i,k_i\rangle \sum_j \langle h_j,b^*e_j\rangle\\
&= \langle (ae_i),k\rangle \langle h,(b^*e_j)\rangle
\end{align*}
and hence
\[
\|T\| = \left(\sum \|ae_i\|^2\right)^{1/2} \left(\sum \|b^*e_j\|^2\right)^{1/2} = \|a\|_2 \|b\|_2.\qquad \qed
\]
\renewcommand{\qed}{}\end{proof}

Let us denote by $C_\lambda^*(G)$ the reduced $C^*$-algebra of a discrete group $G$, i.e. the $C^*$-algebra generated by the left regular representation $\lambda\colon \ G\to B(\ell_2(G))$. 
By
a Herz-Schur multiplier on $C_\lambda^*(G)$, we mean
a bounded linear  map $T$ on $C_\lambda^*(G)$ for which there is a function $f\colon\ G\to \bb C$ such that, for any $t\in G$
$$ T(\lambda(t))=f(t) \lambda(t).$$
 We then denote ${  T}_f=T$. 
 \begin{cor} If $G$ is a finite group and $f$ is the   function
 constantly equal to 1, then
 $\|{T}_f\|_{cob}$ (i.e. the identity map
 on  $C_\lambda^*(G)$)  
 is equal to the supremum of the dimensions
 of the irreducible representations of $G$.
 \end{cor}
 \begin{rem} By Theorem \ref{compthm2}, if $f\equiv 1$ as above, then the function 
 defined on $G\times G$ by ${\hat f}(s,t)=f(st^{-1})$ (that is also constantly equal to 1), 
 satisfies (when viewed as a Schur multiplier on 
 $B(\ell_2(G)$) $\|M_{\hat f}\|_{cob} =|G|$ and in general this is different
 from $\|{T}_f\|_{cob}=\sup \{  \dim(\pi)\mid {\pi\in\widehat G}\}.$\\
Now let $G$ be an infinite discrete group.  By Remarks \ref{rem1} and \ref{rem2}, the identity map on  $A=C_\lambda^*(G)$ is completely co-bounded iff the same is true for $A^{**}$, and this implies that 
 the latter is a  direct sum
of finitely many algebras of the form $M_n \otimes A_n$, with $A_n$ commutative. In particular, of course 
$A^{**}$ is injective, $A$ is nuclear and hence $G$ is amenable.

   \end{rem}

It would be interesting to describe the 
completely co-bounded Herz-Schur multipliers
on the reduced $C^*$-algebra of a discrete group $G$ in analogy
with what is known for the c.b. ones (see \cite{BF}).

\n\textbf{Acknowledgement.} I am grateful to Marek Bo\.zejko and
Marcin Marciniak for stimulating communication.

\end{document}